\documentclass[letterpaper,11pt]{amsart}

\usepackage{amsmath}
\usepackage{amsthm}
\usepackage{mathrsfs}
\usepackage{amssymb,amscd}
\usepackage{dsfont}
\usepackage[all]{xy}
\usepackage{mathtools}
\usepackage{fullpage}
\usepackage{comment}

\usepackage{xcolor}
\usepackage[colorlinks=true,linkcolor=blue,breaklinks,pagebackref]{hyperref}

\usepackage[lite,abbrev,msc-links,alphabetic]{amsrefs}
\newcommand{\fm}{\mathfrak{m}}

\newcommand{\fS}{\mathfrak{S}}

\newcommand{\fL}{\mathfrak{L}}

\newcommand{\C}{\mathbb{C}}

\newcommand{\R}{\mathbb{R}}

\newcommand{\Z}{\mathbb{Z}}

\newcommand{\rd}{\mathrm{d}}

\newcommand{\bs}{\backslash}

\DeclareMathOperator{\diag}{diag} 
 
 \DeclareMathOperator{\mat}{Mat}

\DeclareMathOperator{\GL}{GL}

\DeclareMathOperator{\Ind}{Ind}

\DeclareMathOperator{\Hom}{Hom}

\DeclareMathOperator{\JL}{JL}

\DeclareFontFamily{U}{mathx}{\hyphenchar\font45}
\DeclareFontShape{U}{mathx}{m}{n}{
      <5> <6> <7> <8> <9> <10>
      <10.95> <12> <14.4> <17.28> <20.74> <24.88>
      mathx10
      }{}
\DeclareSymbolFont{mathx}{U}{mathx}{m}{n}
\DeclareMathAccent{\widecheck}{\mathalpha}{mathx}{"71}

\makeatletter
\def\Ddots{\mathinner{\mkern1mu\raise\p@
\vbox{\kern7\p@\hbox{.}}\mkern2mu
\raise4\p@\hbox{.}\mkern2mu\raise7\p@\hbox{.}\mkern1mu}}
\makeatother

\theoremstyle{definition}
\newtheorem{definition}{Definition}[section]
\newtheorem{example}[definition]{Example}

\theoremstyle{plain}
\newtheorem{theorem}[definition]{Theorem}

\newtheorem{lemma}[definition]{Lemma}
\newtheorem{coro}[definition]{Corollary}
\newtheorem{conj}[definition]{Conjecture}

\theoremstyle{remark}
\newtheorem{remark}[definition]{Remark}
\newtheorem*{Ac}{Acknowledgement}

\numberwithin{equation}{section}

\title[Classification of standard modules with linear periods]{Classification of standard modules with linear periods}
\author{Miyu Suzuki}
\address{Miyu Suzuki \\
Faculty of Mathematics and Physics, Institute of Science and Engineering\\
Kanazawa University\\
Kakumamachi, Kanazawa, Ishikawa, 920-1192, JAPAN}
\email{miyu-suzuki@staff.kanazawa-u.ac.jp}

\begin{document}

\maketitle

\begin{abstract}
Suppose that $F$ is a non-Archimedean local field of characteristic not $2$ and $D$ is a central division algebra over $F$.
Let $n$ be a positive integer.
We show a classification modulo essentially square-integrable representations of standard modules of $\GL_n(D)$ which have non-zero linear periods.
By this classification, the conjecture of Prasad and Takloo-Bighash is reduced to the case of essentially square-integrable representations.
\end{abstract}

\section{Introduction}\label{sec:Introduction}

Suppose that $F$ is a non-Archimedean local field of characteristic not $2$ and $E$ is a quadratic extension of $F$.
Take a positive integer $n$ and a central division algebra $D$ over $F$ of dimension $d^2$.
Set $G=G_n=\GL_{n}(D)$, which is the multiplicative group of $\mat_n(D)$.
Let $\nu=\nu_n$ denote the composition of the reduced norm of $\mat_n(D)$ and the normalized absolute value of $F$.
Assume that $E$ embeds in $\mat_n(D)$ and let $H=C_{G}(E^\times)$ be the centralizer of $E^\times$ in $G$.
Note that the centralizer $C_{\mat_n(D)}(E)$ of $E$ in $\mat_n(D)$ is a central simple algebra over $E$.
Let $\chi$ be a character of $E^\times$ and regard it as a character of $H$ by composing with the reduced norm of $C_{\mat_n(D)}(E)$. 
An admissible representation $\pi$ of $G$ is said to be \textit{$(H, \chi)$-distinguished} if $\Hom_H(\pi, \chi)\neq0$.
Prasad and Takloo-Bighash proposed a conjecture \cite[Conjecture 1]{PTB11} about $(H, \chi)$-distinguished representations.
The following is $\chi=1$ case of that conjecture.
\begin{conj}\label{conj:PTB}
Let $\pi$ be an irreducible, admissible representation of $G$ with trivial central character. 
\begin{enumerate}
\item If $\pi$ is $H$-distinguished, 
    \begin{enumerate}
    \item[(i)] The Langlands parameter of $\pi$ takes values in the symplectic group 
    $\mathrm{Sp}_{nd}(\C)$.

    \item[(ii)] The root number satisfies 
        $\varepsilon(\pi)\varepsilon(\pi\otimes\eta_{E/F})
        =(-1)^n\eta_{E/F}(-1)^{nd/2}$.
    \end{enumerate}

\item If $\pi$ is an essentially square integrable representation satisfying (i) and (ii), then $\pi$ is $H$-distinguished.
\end{enumerate}
\end{conj}

\begin{remark}
The original conjecture of Prasad and Takloo-Bighash is stated under the assumption that $\pi$ corresponds to a generic representation of $\GL_{nd}(F)$.
We will clarify the meaning of this assumption in Section \ref{subsec:Jacquet-Langlands transfer}.
One of the new observations of this paper is that this assumption is not necessary.
\end{remark}

This conjecture is based on the result of Tunnell \cite{Tun83} for the case of $D=F$ with odd residual characteristic and $n=2$. 
There are a lot of research to generalize this result.
Saito gave a simpler proof of Tunnell's result  in \cite{Sai93} including the case of residual characteristic $2$ when the characteristic of $F$ is $0$.
Prasad \cite{Pra07} obtained a different proof of Saito's result based on a global theory.

The next three results of \cite{PTB11}, \cite{FMW18} and \cite{Xue} are for the case of characteristic $0$.
Prasad and Takloo-Bighash proved Conjecture \ref{conj:PTB} in \cite{PTB11} for representations of $\GL_2(D)$ where $D$ is a quaternion algebra. 
They used the theta correspondence between $\mathrm{GSO}_6$ and $\mathrm{GSp}_4$.
In \cite{FMW18}, Feigon, Martin and Whitehouse obtained (1) when $D=F$ and $\pi$ is supercuspidal, by using a relative trace formula.
Using the same relative trace formula in a different way, Xue proved (1) for all essentially square-integrable representations and (2) for supercuspidal representations with supercuspidal Jacquet-Langlands transfer to $\GL_{nd}(F)$ (\cite{Xue}).
When $D$ is a quaternion algebra or $D=F$, Xue also proved (2) for all supercuspidal representations.

The following results of \cite{Cho19} and \cite{CM20} cover the case of all non-Archimedean local fields of odd residual characteristic.
In \cite{Cho19}, Chommaux proved the original conjecture (for a general character $\chi$) for Steinberg representations by Mackey machinery.
This result also covers the case where $F$ has characteristic $0$ and the residual characteristic $2$.
Chommaux and Matringe showed the original conjecture when $D=F$ and $\pi$ is a supercuspidal representation of depth zero (\cite{CM20}).

In \cite{Sec20}, S\'echerre proved the conjecture for all supercuspidal representations when $F$ has characteristic $0$  and odd residual characteristic, by using Bushnell-Kutzko type.
Recently, Xue and the author proved that Conjecture \ref{conj:PTB} for essentially square-integrable representations reduces to the case of supercuspidal representations when the characteristic of $F$ is $0$ (\cite{SX20}).
See also \cite[Section 1.7]{Sec20} for a historical review on this conjecture and other related research.

The goal of this paper is to prove the next theorem.

\begin{theorem}\label{thm:classification}
Suppose $P=MU$ is a standard parabolic subgroup of $G$ corresponding to the partition $n=(n_1, \ldots, n_t)$ of $n$.
Let $\widetilde{\pi}=\Delta_1\times\cdots\times\Delta_t$ be a standard module of $G$ induced from segments on $M$.
Then, $\widetilde{\pi}$ is $H$-distinguished if and only if there is an involution $\sigma\in\fS_t$ such that $\Delta_{\sigma(i)}\simeq\Delta_i^\vee$ for all $i$ and if $\sigma(i)=i$, $E$ embeds in $\mat_{n_i}(D)$ and $\Delta_i$ is $C_{G_{n_i}}(E^\times)$-distinguished.
Here, $C_{G_{n_i}}(E^\times)$ is the centralizer of $E^\times$ in $G_{n_i}$.
\end{theorem}

From this, we can reduce Conjecture \ref{conj:PTB} (1) to the case of essentially square-integrable representations.

\begin{theorem}\label{thm:PTB}
Conjecture \ref{conj:PTB} (1) for essentially square-integrable representations implies the general case.
\end{theorem}

Combining this theorem with \cite{Sec20} and \cite{SX20}, we see that Conjecture \ref{conj:PTB} holds when the characteristic of $F$ is zero and  the residual characteristic is different from $2$.
When $D$ is a quaternion algebra or $D=F$,  we can remove the condition on the residual characteristic by applying \cite{Xue} instead of \cite{Sec20}.

\section{Notation}\label{sec:Notation}

Suppose that $F$ is a non-Archimedean local field of characteristic not $2$ and $E$ is a quadratic extension of $F$.
Take a positive integer $n$ and a central division algebra $D$ over $F$ of dimension $d^2$.
Set $G=G_n=\GL_{n}(D)$, which is the multiplicative group of $\mat_n(D)$.
Let $\nu=\nu_n$ denote the composition of the reduced norm of $\mat_n(D)$ and the normalized absolute value of $F$.
Assume that $E$ embeds in $\mat_n(D)$ and let $H=C_{G}(E^\times)$ be the centralizer of $E^\times$ in $G$.
Note that the centralizer $C_{\mat_n(D)}(E)$ of $E$ in $\mat_n(D)$ is a central simple algebra over $E$.

\subsection{Parabolic subgroups}\label{subsec:Parabolic subgroups}

For an ordered partition $\lambda=(n_1, \ldots, n_t)$ of $n$, let $P=P_\lambda$ be the parabolic subgroup of $G$ consisting of matrices of the form
    \[
    \left(\begin{array}{ccc}
    g_1&\ast     &\ast \\
          &\ddots&\ast \\
          &           &g_t \\   
    \end{array}\right)\in G,\quad g_i\in G_{n_i}.
    \]
Write by $M=M_{\lambda}$ the Levi subgroup of $P$ consisting of matrices of the form $\diag(g_1, \ldots, g_t)$ with $g_i\in G_{n_i}$. 
Let $U=U_\lambda$ be the unipotent radical of $P$.
Let $P_0$ be the minimal parabolic subgroup corresponding to the partition $\lambda_0=(1, \ldots, 1)$ and set $M_0=M_{\lambda_0}$.
A parabolic subgroup of $G$ is called \textit{standard parabolic subgroup} if it contains $P_0$.
A standard parabolic subgroup of $G$ has a unique Levi component which contains $M_0$ and is called a \textit{standard Levi component}.
By a parabolic subgroup (resp. Levi component) we shall always mean a standard parabolic subgroup (resp. Levi component).
A Levi subgroup of $G$ will be a Levi component of a parabolic subgroup of $G$.
Throughout the paper, the letters $P$, $M$ and $U$ are reserved for parablic subgroups, their Levi components and their unipotent radicals respectively.
For a Levi subgroup $M$ of $G$, we say that a parabolic subgroup of $M$ is standard if it contains $P_0\cap M$.

Let $P=MU$ be a parabolic subgroup of $G$ and $\rd p$ a left Haar measure on $P$. There is a continuous character $\delta_P\,\colon P\rightarrow\R_{>0}$, called the \textit{modulus character} of $P$, which satisfies
    \[
    \rd(pp')=\delta_P(p')^{-1}\rd p
    \]
for all $p'\in P$.

\subsection{Representations}\label{subsec:Representations}

Let $P=MU$ be a parabolic subgroup of $G$ and $\rho$ an admissible representation of $M$ on a complex vector space $V_\rho$. 
Denote by $\Ind_{P}^G(V_\rho)$ the space of locally constant functions $f\,\colon G\rightarrow V_\rho$ satisfying
    \[
    f(pg)=\delta_P^{1/2}(p)\rho(p)f(g)
    \]
for all $p\in P$ and $g\in G$.
The representation of $G$ on this space given by right translation is called the \textit{parabolically induced representation} from $(\rho, V_\rho)$ and denoted by $(\Ind_P^G(\rho), \Ind_P^G(V_\rho))$.

Let $\lambda=(n_1, \ldots, n_t)$ be the partition corresponding to $P$.
Take a representation $(\rho_i, V_i)$ of $G_{n_i}$ for each $i$ and set $(\rho, V_\rho)=(\rho_1\boxtimes\cdots\boxtimes\rho_t, V_1\otimes\cdots\otimes V_t)$.
This is a representation of $M$.
We extend $\rho$ to a representation of $P$ so that $U$ acts trivially on $V_\rho$ and write the parabolically induced representation $\Ind_{P}^G(\rho)$ by $\rho_1\times\cdots\times\rho_t$.

Let $P'=M'U'$ be a standard parabolic subgroup of $M$.
We denote by $V_\rho(P')$ the subspace of $V_\rho$ spanned by elements of the form $\rho(u')v-v$ with $v\in V_\rho$ and $u'\in U'$.
Set $r_{M', M}(V_\rho)=V_\rho/V_\rho(P')$ and we call this the \textit{Jacquet module} of $V_\rho$ with respect to $P'$.
We write the natural projection $V_\rho\rightarrow r_{M', M}(V_\rho)$ by $j_{P'}$ and define the representation $r_{M', M}(\rho)$ of $M'$ on $r_{M', M}(V_\rho)$ by
    \[
    r_{M', M}(\sigma)(m)j_{P'}(v)=\delta_{P'}^{-1/2}(m')j_{P'}(\sigma(m')v), 
    \quad v\in V,\ m'\in M'.
    \]

We briefly recall basic facts about representations of $G$.
We refer to \cite[Section 3]{LM16} for details.

Let $k$ and $t$ be positive integers such that $n=tk$ and $\rho$ an irreducible supercuspidal representation of $G_k$. 
There is a positive integer $l_\rho$ such that the induced representation $\rho\times\nu^l\rho$ for $l\in\R$ is reducible if and only if $l\in\{\pm l_\rho\}$.
For a pair of integers $(a, b)$ such that $b-a=t-1$, the set 
\[
[a, b]_\rho=\{\rho\nu^{al_\rho}, \ldots, \rho\nu^{bl_\rho}\}
\] 
is called a \textit{segment} of $G$.
Let $\fL([a, b]_\rho)$ denote the unique irreducible quotient of $\rho\nu^{al_\rho}\times\cdots\times\rho\nu^{bl_\rho}$.
For an irreducible essentially square-integrable representation $\pi$ of $G$, there is a unique segment $\Delta$ such that $\pi=\fL(\Delta)$.

Let $\Delta=[a, b]_\rho$ and $\Delta'=[a', b']_{\rho'}$ be two segments.
We say that $\Delta$ and $\Delta'$ are \textit{linked} if $\Delta\cup\Delta'$ is a segment but neither $\Delta\subset\Delta'$ nor $\Delta'\subset\Delta$.
If $\Delta$ and $\Delta'$ are linked and $\rho'\nu^{a'l_{\rho'}}=\rho\nu^{(a+j)l_\rho}$ with some $j>0$, we say that $\Delta$ \textit{precedes} $\Delta'$.
Let $\lambda=(n_1, \ldots, n_t)$ be a partition of $n$ and $\Delta_i$ a segment of $G_{n_i}$.
Let $\fm$ denote the multi-set $\{\Delta_1, \ldots, \Delta_t\}$.
Such a multi-set is called a \textit{multisegment}.
The induced representation $\Delta_1\times\cdots\times\Delta_t$ is called a \textit{standard module} if for every $1\leq i< j\leq t$, $\Delta_i$ does not precede $\Delta_j$.
A standard module has a unique irreducible quotient $\fL(\fm)$.
Note that the isomorphsim class of $\fL(\fm)$ depends only on $\fm$.
For an irreducible admissible representation $\pi$ of $G$, there is a unique multisegment $\fm$ such that $\pi=\fL(\fm)$.

A multisegment $\fm=\{\Delta_1, \ldots, \Delta_t\}$ is said to be \textit{totally unlinked} if $\Delta_i$ and $\Delta_j$ are not linked for all $i, j=1, \ldots, t$.
For a multisegment $\fm=\{\Delta_1, \ldots, \Delta_t\}$ of $G$, the induced representation $\Delta_1\times\cdots\times\Delta_t$ is irreducible if and only if $\fm$ is unlinked.
If this is the case, the induced representation is independent of the ordering of the segments.
When $D=F$, this is equivalent to say that $\fL(\fm)$ is generic.

Let $\widetilde{\pi}=\Delta_1\times\cdots\times\Delta_t$ be a standard module of $G$.
For a supercuspidal representation $\rho$, set $\Z_\rho=\{\nu^{ml_\rho} \mid m\in\Z\}$ and $\pi_\rho=\Delta_{i_1}\times\cdots\times\Delta_{i_l}$, where $\Delta_{i_j}\subset\Z_\rho$ and $\Delta_j\not\subset\Z_\rho$ for all $j\not\in\{i_1, \ldots, i_l\}$.
Reordering the segments, we may write $\widetilde{\pi}\simeq\pi_{\rho_1}\times\cdots\pi_{\rho_p}$.
Further we assume that $\pi_{\rho_i}=[a_1^i, b_1^i]_{\rho_i}\times\cdots\times[a_{q_i}^i, b_{q_i}^i]_{\rho_i}$ with $b_1^i\geq\cdots\geq b_{q_i}^i$ for each $i$.
We call the expression $\pi_{\rho_1}\times\cdots\pi_{\rho_p}$ a \textit{right-ordered form} of the standard module $\widetilde{\pi}$.

\subsection{Jacquet-Langlands transfer}\label{subsec:Jacquet-Langlands transfer}

In this subsection, we summarize basic facts about Jacquet-Langlands transfer.
For details, see \cite[Section 2]{Bad08}.

Jacquet-Langlands transfer is a bijection from the set of irreducible essentially square-integrable representations of $G$ to that of $\GL_{nd}(F)$.
We write this map by $\JL$.
For an irreducible supercuspidal representation $\rho$ of $G_k$, we have $\JL(\rho)=\fL\left(\left[\frac{1-l_\rho}{2}, \frac{l_\rho-1}{2}\right]_{\rho'}\right)$ fore some irreducible supercuspidal representation $\rho'$ of $\GL_{kd/l_\rho}(F)$.

For a segment $\Delta=[a, b]_\rho$ of $G$, let $\JL(\Delta)$ be the segment of $\GL_{nd}(F)$ such that $\JL(\fL(\Delta))=\fL(\JL(\Delta))$.
We have
    \[
    \JL(\Delta)=\left[\frac{(2a-1)l_\rho+1}{2}, \frac{(2b+1)l_\rho-1}{2}\right]_{\rho'}
    \]
and it is easy to check that for two segments $\Delta$ and $\Delta'$, $\Delta$ precedes $\Delta'$ if and only if $\JL(\Delta)$ precedes $\JL(\Delta')$.

Let $\fm=\{\Delta_1, \ldots, \Delta_t\}$ be a multisegment of $G$ and $\pi=\fL(\fm)$ the associated irreducible admissible representation of $G$.
Define a multisegment $\JL(\fm)$ of $\GL_{nd}(F)$ by $\JL(\fm)=\{\JL(\Delta_1), \ldots, \JL(\Delta_t)\}$. 
We set $\JL(\fL(\fm))=\fL(\JL(\fm))$  and call it the \textit{representation of $\GL_{nd}(F)$ corresponding to $\fL(\fm)$}.
In \cite[Section 2.7]{Bad08}, $\JL(\fL(\fm))$ is denoted by $Q_n(\fL(\fm))$.
From the above remark, we see that $\JL(\Delta_1)\times\cdots\times\JL(\Delta_t)$ is a standard module attached to $\JL(\fm)$.
Moreover, the corresponding representation is generic if and only if $\fm$ is totally unlinked.

\section{Proof}\label{sec:Proof}

Let $P=MU$ be a standard parabolic subgroup of $G$ corresponding to the partition $(n_1, \ldots, n_t)$ of $n$. 
We recall some facts on coset representatives for $P\bs G/H$ from \cite{Cho19}.
Let $I(P)$ be the set of symmetric matrices $s=(n_{i,j})\in\mat_t(\Z)$ with non-negative entries such that the sum of $i$-th row equals $n_i$ for each $i=1, \ldots, t$ and when $d$ is odd, all diagonal entries are even.
The set $P\bs G/H$ is in bijection with $I(P)$ and an explicit choice of the coset representative corresponding to $s\in I(P)$, which we write by $u_s$, is given in \cite{Cho19}.
For a subgroup $C$ of $G$, we set $C^{u_s}=C\cap u_sHu_s^{-1}$.
We say that $u_s$ or the coset represented by it is \textit{$P$-admissible} if $s$ is a monomial.
Take $s=(n_{i, j})\in I(P)$ and set $u=u_s$.
Let $P_s=M_sU_s$ be the standard parabolic subgroup of $G$ corresponding to the partition $(n_{1,1}, n_{1,2}, \ldots, n_{t, t-1}, n_{t,t})$ of $n$.
Note that $n_{i,j}$ could be $0$ and such entries are ignored when we say $(n_{1,1}, n_{1,2}, \ldots, n_{t, t-1}, n_{t,t})$ is a partition.
Let $\widetilde{\pi}=\Delta_1\times\cdots\times\Delta_t$ be a standard module of $G$ induced from a segment $\Delta=\Delta_1\boxtimes\cdots\boxtimes\Delta_t$ on $M$.

For $s=(n_{i, j})\in I(P)$, the Jacquet module $r_{M, M_s}(\Delta)$ of $\Delta$ with respect to the standarad parabolic subgroup $P_s\cap M=M_s(U_s\cap M)$ of $M$ is written as 
    \[
    \Delta_{1,1}\boxtimes\Delta_{1,2}\boxtimes\cdots
    \boxtimes\Delta_{t, t-1}\boxtimes\Delta_{t, t},
    \]
where $\Delta_{i, j}$ is a segment on $G_{n_{i, j}}$ for $i, j=1, \ldots, t$ such that $\Delta_i=\Delta_{i, 1}\sqcup\cdots\sqcup\Delta_{i, t}$ with the convention that $G_{n_i,j}=\{1\}$ and $\Delta_{i, j}$ is the trivial representation if $n_{i,j}=0$.

The following is \cite[Proposition 5.5]{BM19}, which is a consequence of the Mackey theory.
\begin{lemma}\label{lem:Mackey}
Let $\widetilde{\pi}=\Delta_1\times\cdots\times\Delta_t$ be a standard module of $G$ induced from segments on $M$.
Suppose that $\widetilde{\pi}$ is $H$-distinguished.
Then there is $s\in I(P)$ such that the Jacquet module $r_{M, M_s}(\Delta)$ on $M_s$ is $M_s^{u_s}$-distinguished.
If we write $r_{M, M_s}(\Delta)$ as above, this is equivalent to $\Delta_{i,j}\simeq\Delta_{j, i}^\vee$ for all $i, j$ and $\Delta_{i,i}$ is $C_{G_{n_{i,i}}}(E^\times)$-distinguished for all $i$.
\end{lemma}

By a standard argument, we can show that only $P$-admissible cosets contribute to $H$-invariant linear forms.

\begin{lemma}\label{lem:admissible}
Let $\widetilde{\pi}=\Delta_1\times\cdots\times\Delta_t$ be a right-ordered form of a standard module of $G$ induced from segments on $M$.
If $u=u_s$ is not $P$-admissible, then the Jacquet module $r_{M, M_s}(\Delta)$ is not $M_s^u$-distinguished.
\end{lemma}

\begin{proof}
This result is stated in \cite[Lemma 3.3]{Gur15} for Galois involution, which is a generalization of \cite{Mat11}.
The proof holds almost verbatim for our case.
We include the proof for completeness.

Assume that $u=u_s$ is not $P$-admissible and $r_{M, M_s}(\widetilde{\pi})$ is $M_s^u$-distinguished.
Since $u$ is not $P$-admissible, $M_s$ is strictly contained in $M$.
There exists $1\leq i\leq t$ such that $\{n_{i, j} \mid j=1, \ldots, t\}-\{0\}$ is not a singleton.
Let $i_0$ be the minimal such index and we write $\{j\mid n_{i_0, j}\neq0\}=\{ j_1<j_2<\cdots< j_r\}$.
We can write $\Delta_{i_0}=[a, b]_\rho$ with some supercuspidal representation $\rho$ and integers $a<b$.
There are some integers $c$ and $d$ such that $a\leq c\leq d<b$ and $\Delta_{i_0, j_1}=[d+1, b]_\rho$ and $\Delta_{i_0, j_2}=[c, d]_\rho$.
Note that we have $\Delta_{j_1, i_0}\simeq\Delta_{i_0, j_1}^\vee\simeq[-b, -d-1]_{\rho^\vee}$ and $\Delta_{j_2, i_0}\simeq\Delta_{i_0, j_2}^\vee\simeq[-d, -c]_{\rho^\vee}$.
Since $\Delta=\Delta_1\times\cdots\times\Delta_t$ is a right-ordered form, $\Delta_{j_1}\simeq[a', b']_{\rho^\vee}$ for some integers $a'\leq -b$ and $b'\geq -c\geq-d>-d-1$.
In particular, $i_0>1$.
Write $\Delta_{j_1,1}\simeq[e, b']_{\rho^\vee}$ with some $-d\leq e\leq b'$. 
Then we get $\Delta_{1, j_1}\simeq[-b', -e]_{\rho}$  and since $i_0>1$, we have $\Delta_1=\Delta_{1, j_1}$.
This contradicts the fact that $\Delta=\Delta_1\times\cdots\times\Delta_t$ is a right-ordered form since $-e<b$.
\end{proof}

The main theorem is a straightforward consequence of Lemma \ref{lem:admissible}.

\begin{proof}[Proof of Theorem \ref{thm:classification}]
Suppose that $\Delta$ is $H$-distinguished.
From Lemma \ref{lem:Mackey}, there is $s\in I(P)$ such that $r_{M, M_s}(\Delta)$ is $M_s^{u_s}$-distinguished.
By Lemma \ref{lem:admissible}, $u=u_s$ is $P$-admissible.
In other words, $s$ is a monomial matrix in $\mat_t(\Z_{\geq0})$.
Let $\sigma\in\fS_t$ be the corresponding permutation.
Then $\sigma$ is an involution and satisfies the desired conditions.

By the same argument as the proof of \cite[Proposition 2.9]{MO18}, we obtain the converse implication.
\end{proof}

Now we deduce Theorem \ref{thm:PTB} from Theorem \ref{thm:classification}.

\begin{proof}[Proof of Theorem \ref{thm:PTB}]
Let $\widetilde{\pi}=\Delta_1\times\cdots\times\Delta_t$ be a standard module of $G$ which has a surjection onto $\pi$.
Since $\pi$ is $H$-distinguished, so is $\widetilde{\pi}$.
Let $\sigma\in\fS_t$ be the permutation of Theorem \ref{thm:classification}.
It is easy to check that the $L$-parameter of $\pi$ takes values in the symplectic group.
What we have to show is the equation of Conjecture \ref{conj:PTB} (ii).
We have
    \[
    \varepsilon(\pi)\varepsilon(\pi\otimes\eta_{E/F})
    =\prod_{i=1}^t\varepsilon(\Delta_i)\varepsilon(\Delta_i\otimes\eta_{E/F}).
    \] 
Suppose that $\Delta$ is $H$-distinguished.
Let $\sigma\in\fS_t$ be the involution in Theorem \ref{thm:classification}.
For each $i$, let $\omega_i$ be the central character of $\Delta_i$.
Since $\Delta_{\sigma(i)}\simeq\Delta_i^\vee$, we have 
    \[
    \varepsilon(\Delta_i)\varepsilon(\Delta_{\sigma(i)})=\omega_i(-1), \quad
    \varepsilon(\Delta_i\otimes\eta_{E/F})
    \varepsilon(\Delta_{\sigma(i)}\otimes\eta_{E/F})=\omega_i(-1)\eta_{E/F}(-1)^{n_id}.
    \]
Hence we get
    \[
    \varepsilon(\pi)\varepsilon(\pi\otimes\eta_{E/F})
    =\prod_{i;\, \sigma(i)<i}\eta_{E/F}(-1)^{n_id}\prod_{i;\, \sigma(i)=i}\varepsilon(\Delta_i)
    \varepsilon(\Delta_i\otimes\eta_{E/F}).
    \]
Let $i$ be a fixed point of $\sigma$. 
Since $\Delta_i$ is $C_{G_{n_i}}(E^\times)$-distinguished, we have 
    \[
    \varepsilon(\Delta_i)\varepsilon(\Delta_i\otimes\eta_{E/F})
    =(-1)^{n_i}\eta_{E/F}(-1)^{n_id/2}
    \] 
from the essentially square-integrable case. 
Therefore the above product equals $(-1)^n\eta_{E/F}(-1)^{nd/2}$.
\end{proof}

\begin{coro}
\begin{enumerate}
\item Conjecture \ref{conj:PTB} holds without the condition that $\pi$ corresponds to 
    a generic representation of $\GL_{nd}(F)$.

\item For representations corresponding to generic representations, 
    \textit{i.e.} representations induced from totally unlinked multisegments, 
    Theorem \ref{thm:classification} provides a complete classification of 
    $H$-distinguished representations modulo essentially square-integrable 
    representations.
\end{enumerate}
\end{coro}

\begin{proof}
First we consider the first assertion.
In the proof of Theorem \ref{thm:PTB}, we did not use the assumption that $\pi$ corresponds to a generic representation.
Hence the first part of Conjecture \ref{conj:PTB} holds for any irreducible admissible representations.
Note that Jacquet-Langlands transfer of essentially square integrable representations are essentially square integrable, hence generic.
Therefore the assumption in the original conjecture of Prasad and Takloo-Bighash is unnecessary.

The second assertion is obvious from the fact that standard module is irreducible if and only if it is induced from totally unlinked multisegments.
\end{proof}

\begin{Ac}
I would like to thank Hang Xue  for constant support and helpful discussion. 
In the previous version of this paper, I only treated representations induced from totally unlinked multisegments.
I am very grateful to Nadir Matringe who suggested considering the general standard modules by the argument of \cite{Gur15}.
I am also grateful to Dipendra Prasad for answering many questions and  Alberto M\'ingeuz for useful comments. 
I appreciate the anonymous referee for careful reading and pointing out several mistakes.
This work is partially supported by JSPS Research Fellowship for Young Scientists No. 20J00434.
\end{Ac}

\end{document}